\documentclass[reqno]{amsproc}  
\usepackage[dvips]{graphicx}     
\usepackage{amsmath,amsthm,amssymb,cases}

\newtheorem{prop}{\bf Proposition}[section]
\newtheorem{thm}[prop]{\bf Theorem}
\newtheorem{lem}[prop]{\bf Lemma}
\newtheorem{ex}[prop]{\bf Example}
\newtheorem{defi}[prop]{\bf Definition}
\newtheorem{rmk}[prop]{\bf Remark}

\makeatletter
\@addtoreset{equation}{section}

\makeatother

\def\R{\mathbb{R}}
\def\C{\mathbb{C}}
\def\Z{\mathbb{Z}}

\def\bd{{\partial}}
\def\vol{\text{Vol} }

\def\ms{\medskip\noindent}
\def\ss{\smallskip\noindent}

\title{On the number of hyperbolic $3$-manifolds of a given volume}
\author[Hodgson]{Craig Hodgson}
\address{Department of Mathematics and Statistics,
University of Melbourne, Parkville, Victoria 3010, Australia}
\email{craigdh at unimelb.edu.au}
\author[Masai]{Hidetoshi Masai}
\address{Department of Mathematical and Computing Sciences, Tokyo Institute of
Technology, O-okayama, Meguro-ku, Tokyo 152-8552 Japan}
\email{masai9 at is.titech.ac.jp}

\date{}

\begin{document}

\thispagestyle{headings}

\begin{abstract}
The work of J{\o}rgensen and Thurston shows that there is a finite number $N(v)$ 
of orientable hyperbolic 3-manifolds with any given volume  $v$.   
We show that there is an infinite sequence of closed orientable hyperbolic 3-manifolds,
obtained by Dehn filling on the figure eight knot complement, 
that are uniquely determined by their volumes. This gives a sequence of distinct 
volumes $x_i $ converging to the volume of the figure eight knot complement 
with $N(x_i) = 1$ for each $i$.  We also give
an infinite sequence of 1-cusped hyperbolic 3-manifolds, obtained
by Dehn filling one cusp of  the $(-2,3,8)$-pretzel link complement, 
that are uniquely determined by their volumes amongst orientable cusped 
hyperbolic 3-manifolds. 
Finally, we describe examples showing that the number of hyperbolic link complements
with a given volume $v$ can grow at least exponentially fast with $v$.
\end{abstract}

\maketitle

\section{Introduction}
Thurston and J{\o}rgensen (see \cite{Th}) showed that the set of volumes of finite volume 
orientable complete hyperbolic $3$-manifolds is a closed, non-discrete, well ordered 
subset of $\R_{>0}$ of order type $\omega^\omega$. 
Further, the number $N(v)$ of complete orientable hyperbolic $3$-manifolds of volume $v$ is finite 
for all $v\in \mathbb{R}_{>0}$. In this paper we study how $N(v)$ varies with $v$.

There has been much work on determining the lowest volumes in various classes of
hyperbolic 3-manifolds.
The work of Gabai-Meyerhoff-Milley (\cite{GMM1,GMM2,M}) 
shows that the Weeks manifold  
with volume $v_1 = 0.9427 \ldots$ is the unique orientable 
hyperbolic 3-manifold of lowest volume, so that $N(v_1)=1$.
Prior to our work,  this was the only known exact value for $N(v)$.
 
 It is also interesting  to look at the number of hyperbolic 3-manifolds of a given volume in various special classes.
For example, we could study the numbers $N_C (v)$, $N_A(v)$, $N_L(v)$, $N_G(v)$
of orientable hyperbolic 3-manifolds of volume $v$ that are cusped, arithmetic, link complements, or with
geodesic boundary, respectively. 

The work of Cao-Meyerhoff \cite{CM} shows there are precisely two orientable cusped hyperbolic 3-manifolds
of lowest volume $v_\omega \approx 2.029883\ldots$, the figure eight knot complement 
and its sister (obtained by (-5,1) surgery on the Whitehead link complement). 
(See Figure \ref{fig8_fig8sister}.) 

\begin{figure}[h]
\hspace{0.5cm}
\includegraphics[scale=1]{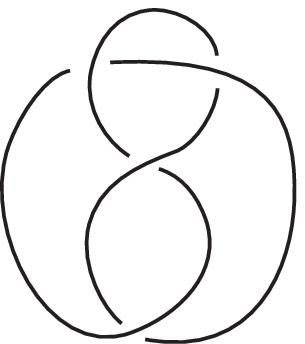}\hspace{2cm} \includegraphics[scale = 1]{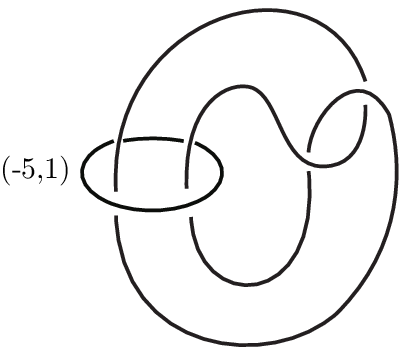}
\caption{The figure eight knot and its sister, with complements denoted $m004$ and $m003$ respectively in SnapPea notation.}
\label{fig8_fig8sister}
\end{figure}

Gabai-Meyerhoff-Milley \cite{GMM1} determined the  first 10 low volume orientable cusped hyperbolic 3-manifolds.
This shows that the next 4 limit volumes are:
$$ v_{2\omega} = 2.568970\ldots, v_{3\omega}=2.666744\ldots ,
v_{4\omega}= 2.781833\ldots, v_{5\omega} = 2.828122\ldots$$
and that
$N_C(v_\omega) = 2 ,  N_C(v_{2\omega})=2,  N_C(v_{3\omega})=2, 
N_C(v_{4\omega})=1, N_C(v_{5\omega})=3.$

The work of Agol \cite{agol_2cusp} shows that there are exactly two orientable 2-cusped hyperbolic 
3-manifolds of lowest volume $v_{\omega^2} = 3.663862\ldots$, namely the Whitehead link complement
and the $(-2,3,8)$-pretzel link complement. (See Figure \ref{fig.WL_WS}.)

\begin{figure}[h]
\includegraphics[scale=0.95]{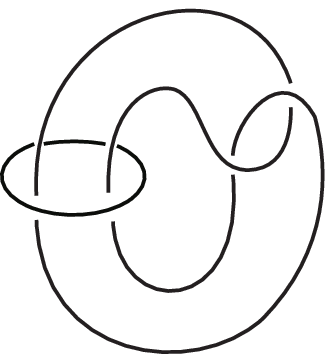}\hspace{2cm} \includegraphics[scale = 1.1]{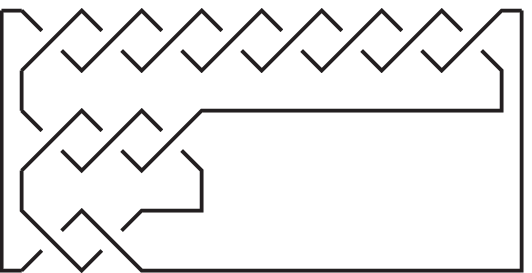}
\caption{The Whitehead link and the $(-2,3,8)$-pretzel link, with complements $m129$ and $m125$ in SnapPea notation.}
\label{fig.WL_WS}
\end{figure}

Chinburg-Friedman-Jones-Reid  \cite{CFJR} showed that
 the Weeks manifold and the Meyerhoff manifold are the unique 
 arithmetic hyperbolic 3-manifolds of lowest volume; hence
 $N_A(0.942707\ldots) = 1$ and $N_A(0.981368\ldots) = 1$.

Kojima and Miyamoto \cite{KojMiy,miyamoto} showed the lowest volume for
a compact 
hyperbolic 3-manifold with geodesic boundary is $6.451990\ldots$,  and 
Fujii \cite{Fu} showed that there are exactly 8 such manifolds with this volume.

\bigskip
The graphs in Figures \ref{closed_cusped_vols} and \ref{cusped_vols} below
show the frequencies of volumes arising from the manifolds
contained in the Callahan-Hildebrand-Weeks census of  
cusped hyperbolic 3-manifolds (see \cite{CaHiWe}),
 and the Hodgson-Weeks census of closed hyperbolic 3-manifolds (see \cite{HW}).
These give {\em lower bounds} on $N(v)$, but note that  infinitely many manifolds
 are definitely missing in these pictures. Nevertheless, the results
 suggest that there are often very few manifolds with any given volume.

\begin{figure}[h]
\includegraphics[scale=1]{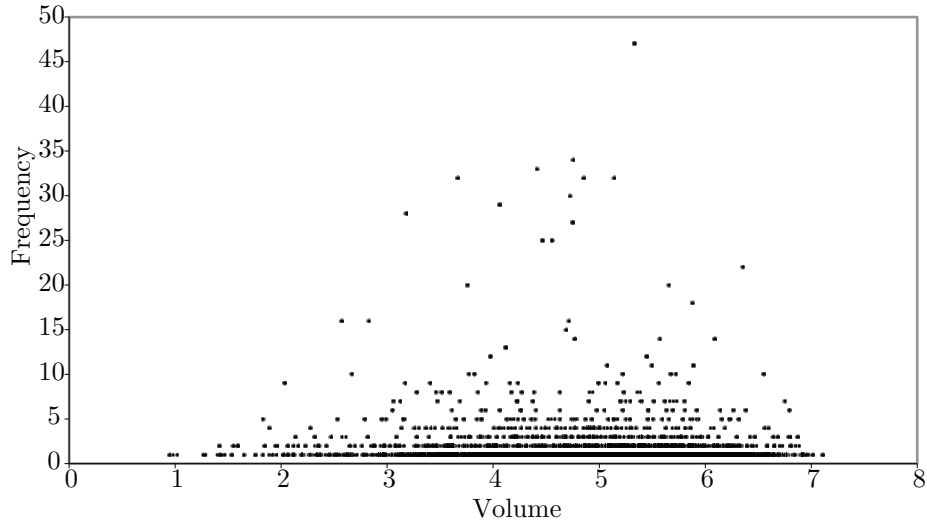}
\caption{Frequencies of volumes of hyperbolic 3-manifolds  in the closed and cusped censuses.}
\label{closed_cusped_vols}
\end{figure}

\begin{figure}[h]
\includegraphics[scale=1]{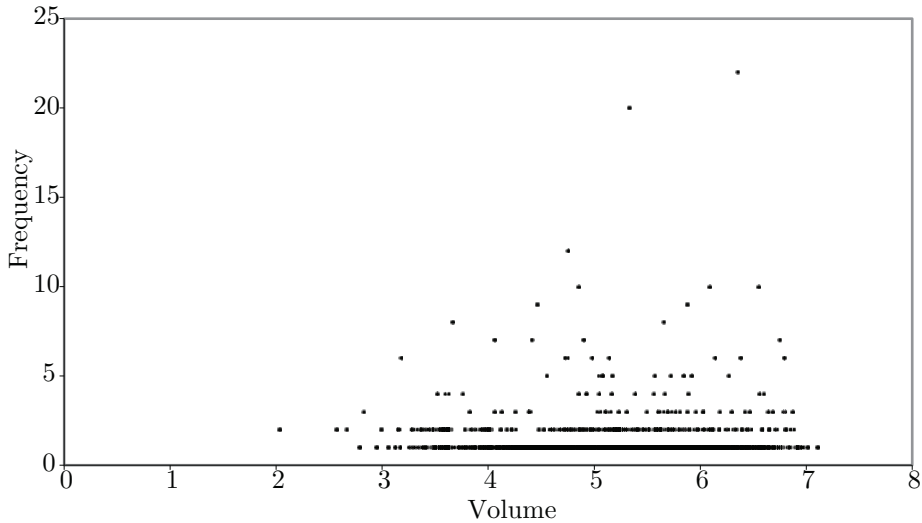}
\caption{Frequencies of volumes of hyperbolic 3-manifolds in the cusped census.}
\label{cusped_vols}
\end{figure}

\medskip
 The first main result of this paper (Theorem \ref{closed_unique}) shows that there is an infinite sequence of closed hyperbolic 3-manifolds determined by their volumes, obtained by Dehn fillings on the figure eight knot complement.
 Hence $N(x_i)=1$ for an infinite sequence of volumes $x_i$ converging to $v_\omega$ from below.

The proof of Theorem \ref{closed_unique} uses the asymptotic formula of Neumann-Zagier \cite{NZ} for the change in volume during hyperbolic Dehn filling and a study of the quadratic form giving the squares of lengths of closed geodesics on a horospherical torus cusp cross section, together with some elementary number theory.

\medskip
Our second major result (Theorem \ref{cusped_unique}) 
shows that there is an infinite sequence of 1-cusped hyperbolic 3-manifolds determined by their volumes amongst orientable  cusped  hyperbolic 3-manifolds. 
These manifolds are obtained by Dehn fillings on one cusp of the $(-2,3,8)$-pretzel link complement. Hence $N_C(y_i)=1$ for an infinite sequence of volumes $y_i$ converging to $v_{\omega^2}$ from below.

The proof of Theorem \ref{cusped_unique} is similar to the proof of Theorem \ref{closed_unique}, 
but there are some additional
complications that require us to study higher order terms in the Neumann-Zagier asymptotic formula.

\medskip
Another obvious feature of the graphs in Figures \ref{closed_cusped_vols} and \ref{cusped_vols} 
is that there are special values of $v$ for which $N(v)$ is particularly large. 
So it is natural to ask how fast $N(v)$ can grow with $v$.

If we fix a natural number $n$,  then Wielenberg \cite{Wie} showed that there exists $v$ such that $N(v)>n$,
and  Zimmerman  \cite{Zim} constructed $n$ distinct closed hyperbolic $3$-manifolds 
which share the same volume. In both results, 
the number of manifolds of volume $v$ constructed was at most linear  in $v$.

More generally, by considering covering spaces of a fixed non-arithmetic 
hyperbolic 3-manifold whose fundamental group surjects onto a free group of rank 2,
it is possible to find sequences of volumes $x_i \to \infty$ and a constant $c >0$
such that $N(x_i) > x_i^{c \,x_i}$ for all $i$. This argument, known to Thurston 
and Lubotzky, can be found in \cite[pp2633-2634]{carlip} or \cite[p1162]{BGLM}. 

Recent work of  Belolipetsky-Gelander-Lubotzky-Shalev \cite{BGLS} 
gives very  precise upper and lower bounds on the number $N_A(v)$ of
{\em arithmetic} hyperbolic 3-manifolds of volume $v$: there exist constants
$a,b>0$ such that for $x >>0$, 
$$ x^{ax} \le \sum_{x_i \le x} N_{A}(x_i) \le  x^{bx}.$$
In particular, this implies that for all sufficiently large $x$,
$$ x^{ax} \le ~\max_{x_i \le x} N_{A}(x_i)~ \le  x^{bx}.$$
 
In \cite{FMP2} (and \cite{FMP1}), Frigerio, Martelli and Petronio constructed explicit families
of hyperbolic $3$-manifolds with totally geodesic boundary 
giving sequences $x_i \to \infty$ such that the number $N_{G} (x_i)$ of hyperbolic 3-manifolds of volume $x_i$ with geodesic boundary grows at least as fast as $x_i^{c x_i}$ for some $c>0$.

\medskip
In the last part of this paper, we investigate
the number $N_L(v)$ of hyperbolic link complements in $S^3$ with volume $v$.
In Theorem \ref{thm.expo} 
we construct sequences of volumes  $v_n \to \infty$ 
such that the number $N_L(v_n)$ of hyperbolic link complements with volume $v_n$ grows at least exponentially fast with $v_n$. 

Our main tool is mutation along totally geodesic thrice punctured spheres in fully augmented link complements.
It is well known that mutation along certain special surfaces leaves the hyperbolic volume
 and many other invariants unchanged (see for example \cite{Ada}, \cite{Rub},  \cite{DGST}).
 We distinguish the manifolds produced using information 
 about their cusp shapes and maximal horoball cusp neighbourhoods. 

Recent independent work of Chesebro-DeBlois \cite{CD} uses 
mutations along 4-punctured spheres to show that  $N_L(v)$ grows 
at least exponentially fast with $v$, and that the number of commensurability classes
of link complements with a given volume can be arbitrarily large.

Mutation along thrice punctured spheres can be studied by using spatial trivalent graphs. 
We will say that a spatial trivalent graph is {\em hyperbolic} if its complement admits a hyperbolic metric with parabolic meridians and geodesic boundaries (see section \ref{sec.Vol}, \cite{HHMP} and \cite{Mas}).
In \cite{Mas}, the second author exhibited volume preserving moves on hyperbolic graphs and hyperbolic links, which correspond to cutting and gluing along thrice punctured spheres.
By applying these moves, each hyperbolic graph with a planar diagram corresponds to a {\em fully augmented link}.
Such links have been well studied, and we can easily compute the moduli of their cusps (see \cite{FP} and \cite{Pur}).

\medskip
Here is a brief outline of the remainder of the paper.  
In section 2, we recall the results of Neumann and Zagier relating the decrease in volume during 
hyperbolic Dehn filling to the length of the surgery curve. We then explain the key ideas that
will allow us to obtain manifolds uniquely determined by their volumes. 
In section 3 (resp. 4), we apply these ideas to the manifolds denoted $m003$ and $m004$ (resp. $m125$ and $m129$) in
SnapPea notation, and produce infinitely many manifolds that are determined by their volumes
amongst all (resp. all cusped) orientable hyperbolic 3-manifolds.
In section 5, we compute some higher order terms in  the Neumann-Zagier
asymptotic formula for the examples considered in sections 3 and 4;
some of those computations are used in section 4.
In section \ref{sec.Vol}, we review some of the volume preserving moves on graphs from \cite{Mas}. 
In section \ref{pack}, we compute the moduli of the cusps of the manifolds that are obtained 
from the graph in Figure \ref{fig.graphW}. 
Then in section \ref{const}, we construct manifolds corresponding to binary cyclic words of $n$ letters and show that 
these manifolds are isometric if and only if their cyclic words are related by the natural action of the dihedral group $D_n$. 
Since the volume of these manifolds is a constant multiple of $n$, 
this gives a sequence $\mathcal{L}_n$ of families of distinct hyperbolic link complements which share the same 
volume $v_n$, such that the growth rate of $\#\mathcal{L}_n$ is at least exponential in $v_n$.
Finally,  in section 9, we collect some open problems related to the work in this paper.

\medskip\noindent
{\bf Acknowledgments:}
We thank Sadayoshi Kojima and Danny Calegari for asking questions that led to this paper.
We also thank Alan Reid,  Rob Meyerhoff, Ze'ev Rudnick  and Peter Sarnak for their helpful comments.
The work of the first author was partially supported by the Australian Research Council grant DP1095760.
The work of the second author was partially supported by the Global COE program
 ``Computationism as a Foundation for the Sciences".
Much of this work was done during the second author's visit to the University of Melbourne; this visit was supported by the above Global COE program.

\section{Hyperbolic Dehn fillings determined by their volumes}\label{cusped_hyp}

Let $M$ be a cusped hyperbolic 3-manifold and let $T$ be a horospherical
torus cross section of one of the cusps of $M$.
Choose an oriented basis for $H_1(T;\Z) \cong \Z^2$,  called the ``meridian'' and ``longitude'' for $T$, and 
let $M(a,b)$ be the result of Dehn filling along the 
simple closed curve with homology class $(a,b) \in H_1(T;\Z) \cong \Z^2$.

Let $A$ denote the area of $T$ and $L(a,b)$ the length of the geodesic on $T$ with 
homology class $(a,b)$. Then we will study the quadratic form
\begin{equation}\label{eq1}
Q(a,b) = {L(a,b)^2 \over A}.
\end{equation}
(This gives  the square of the length of  $(a,b)$ when the torus $T$ is rescaled to have area 1,
and is also known as the extremal length of the curve $(a,b)$.)

\begin{rmk}(Cusp Shape Parameters)\label{cusp_param}
{\em 
The horospherical torus $T$ is similar to the Euclidean torus  $\C/(\Z+\tau\Z)$, 
with {\em shape parameter} $\tau \in \C$ defined by
\begin{equation}\label{eq2}
\tau = (\text{complex translation of longitude})/
(\text{complex translation of meridian}),
\end{equation}
where the torus is viewed from the cusp. 
Note that with the standard orientation 
convention  for a link complement in $S^3$ or
for an oriented manifold in SnapPea, this means that $\tau$ will
have {\em negative} imaginary part.\footnote{Caution: this convention
agrees with that used in Snap. However, in SnapPy the `shape' reported 
by the function {\tt cusp\_info()}
is the {\em complex conjugate} $\bar \tau$.}
Then we have
\begin{equation}\label{eq3}
Q(a,b) = {  |a + b \tau|^2 \over | \text{Im}( \tau )| }.
\end{equation} 
}
\end{rmk}

\medskip
The following result shows that the volume of $M(a,b)$ determines 
$Q(a,b)$ up to a uniformly bounded error.
This observation will play a key role in our work.

\begin{prop}\label{error_bound}  Let $V_0$ be the volume of the complete hyperbolic structure on $M$.
Then there exist constants $C_1=C_1(M), C_2 =C_2(M) >0$ 
such that
\begin{enumerate}
\item[(a)] the Dehn filling $M(a,b)$ is hyperbolic whenever  $L(a,b) > C_1$, and
\item[(b)] for all hyperbolic Dehn fillings $M(a,b)$,
\begin{equation}\label{eq4}
\left| {\pi^2  \over \Delta V(a,b)}  - {Q(a,b)} \right| < C_2
\end{equation} 
where $\Delta V(a,b) = V_0 - \vol(M(a,b))$.
\end{enumerate}
\end{prop}

\begin{proof}

(a) follows from Thurston's hyperbolic Dehn surgery theorem \cite{Th}. 

(b)  The asymptotic volume formula  of Neumann-Zagier \cite{NZ} shows that the
decrease in volume under Dehn filling satisfies
\begin{equation}  \label{eq5}
\Delta V(a,b) = {\pi^2 \over Q(a,b)} +O\left({1 \over Q(a,b)^2}\right).
\end{equation}
This implies the result of (b) when $Q(a,b)$ is sufficiently large.
But there are only finitely many additional $(a,b)$ with $M(a,b)$ hyperbolic, so the result follows.
\end{proof}

\begin{rmk}
{\em
(a) If $T$ is a maximal embedded
horospherical torus, then we can take $C_1=6$ by the ``length 6 theorem'' of Agol \cite{agol_dehn} and 
Lackenby \cite{lack} and the Geometrization Theorem of Perelman.

(b) The work of Hodgson-Kerckhoff \cite[Theorem 5.12 and Figure 2]{HK}, 
and some numerical calculations  imply  that
\begin{equation}  \label{eq6}
-7.05 \le {\pi^2  \over \Delta V(a,b)} - Q(a,b)  \le 5.82
\end{equation}
whenever $Q(a,b) \ge 57.5041$ and $\Delta V \le 0.155$.
Then by volume computations using SnapPy or Snap for the finite number of hyperbolic Dehn fillings
with $\Delta V > 0.155$, an explicit value of $C_2$ can be determined for a given manifold $M$.

Numerical calculations using SnapPy suggest that 
$\left| {\pi^2  \over \Delta V(a,b)} - Q(a,b)\right|$
for hyperbolic Dehn fillings $M(a,b)$ is often considerably  lower
than these general bounds.}
\end{rmk}

\smallskip
We now study the quadratic form $Q(a,b)$ for a cusped hyperbolic 3-manifold $M$. Let
\begin{equation}  \label{eq7}
S_Q= \{ Q(a,b) : a,b \text{ are relatively prime integers} \}.
\end{equation}

\ss
{\bf Key Idea.} Let  $(a_0,b_0)$ be a pair of relatively prime integers such that
$M(a_0,b_0)$ is hyperbolic and let $q_0 = Q(a_0,b_0)$. Assume that 
\begin{enumerate}
\item[(i)]  there is large enough 2-sided gap around  $q_0$ in $S_Q$, and 
\item[(ii)] $Q(a,b)=q_0$  has few solutions with $(a,b)$ relatively prime integers.
\end{enumerate}
Then there are few Dehn fillings $M(a,b)$ with the same volume as 
$M(a_0,b_0)$.

\newpage
More precisely, 
\begin{prop} \label{few_mflds_with_vol} 
Let  $(a_0,b_0)$ be a pair of relatively prime integers such that
$M(a_0,b_0)$ is hyperbolic and let $q_0 = Q(a_0,b_0)$.
Let $C_2 = C_2(M) >0$ be as in Proposition \ref{error_bound}.
Assume that 
\begin{center}
there exists $g > 2C_2$ such that $s \in S_Q$ and $|s-q_0|< g$ implies $s= q_0$.
\end{center}
Then $\vol(M(a,b)) =  \vol(M(a_0,b_0)$ implies $Q(a,b)=Q(a_0,b_0)$.
Hence the number of such manifolds $M(a,b)$  
is at most  $n(q_0)/2$, where $n(q_0)$  is the number of solutions to $Q(a,b)=q_0$ with $a,b$ relatively prime integers.
 \end{prop}

\begin{proof}
Assume that $M(a,b)$ is hyperbolic with $\vol(M(a,b)) =  \vol(M(a_0,b_0))$
 and let $\Delta V=  V_0 - \vol(M(a_0,b_0))$. 
 Then Proposition \ref{error_bound} implies 
$$\left| {\pi^2  \over \Delta V}  -Q(a,b) \right| < C_2 \text{ and } \left| {\pi^2  \over \Delta V}  -Q(a_0,b_0) \right| < C_2,$$
so 
 $$\left| Q(a,b)   - Q(a_0,b_0) \right| < 2C_2 < g.$$
  Hence $Q(a_0,b_0)=Q(a,b)$ by our assumption. 
  But $(a,b)$ and $(-a,-b)$ Dehn fillings on $M$ give the same manifold, so the result follows.
 \end{proof}
 
  \begin{rmk}  \label{few_with_symm} 
 {\em
More generally, there is a version of the previous Proposition which applies 
when the horospherical torus $T = \C/(\Z+\tau \Z)$ has symmetries other than those 
induced by the maps $(a,b) \mapsto \pm(a,b)$ that
 extend to symmetries of the cusped manifold $M$.

Let $G_Q$ be the group of integral automorphisms of the quadratic form  $Q$, 
 \begin{equation}  \label{eq8}
 G_Q = \{ A \in GL(2,\Z) :  Q(A X )= Q( X  ) \text{ for all } X=(x,y) \in \Z^2 \}.
 \end{equation}
For each isometry $g$ of $M$ preserving the cusp corresponding to $T$, the induced action on
homology gives an automorphism $g_*$ of $H_1(T) \cong \Z^2$. Let $G_M \subset G_Q$
be the subgroup generated by these automorphisms $g_*$ together the automorphism 
$(x,y) \mapsto (-x,-y)$.  Then for each $\sigma \in G_M$, the Dehn fillings of $M$ along
$(a,b)$ and $(a',b')=\sigma(a,b)$ give homeomorphic manifolds. 

Then the previous argument shows that, with the hypotheses and notation of Proposition \ref{few_mflds_with_vol},
the number of hyperbolic manifolds $M(a,b)$ with  $\vol(M(a,b)) =  \vol(M(a_0,b_0)$ 
is at most  $n(q_0)/|G_M|$, where $n(q_0)$  is the number of solutions to 
$Q(a,b)=q_0$ with $a,b$ relatively prime integers.
} \end{rmk}

\bigskip
If the cusp shape parameter $\tau$ belongs to an 
imaginary quadratic field, then the quadratic form 
$$Q(x,y) = |x+\tau y|^2= (x + \tau y)(x + \bar\tau y) = x^2 + (\tau+\bar\tau) xy + |\tau|^2 y^2$$
has rational coefficients. Hence, after rescaling, we can assume that
the coefficients are relatively prime integers.

If $Q$ is an {\em integral quadratic form} 
\begin{equation}  \label{eq9}
Q(x,y)= a x^2 + b xy + c y^2.
 \end{equation}
where $a,b,c \in \Z$, then $D = b^2 - 4ac$ is called
the {\em discriminant} of $Q$. 
The following elementary result can be found, for example, in 
\cite[Theorem 201, p180]{landau}, \cite[pp25--26]{cox} or \cite[pp49--50]{buell}.
Here, and henceforth,  we say that a pair of integers $(x,y)$ is {\em primitive} if $\gcd(x,y)=1$.
 
\begin{lem} \label{discr_quad_res} 
Assume that $m$ has a {primitive} integer representation by 
an integral  quadratic form $Q$, i.e.
$Q(x_0,y_0)=m$ where $(x_0,y_0) \in \Z^2$ with $\gcd(x_0,y_0)=1$.
Then the discriminant $D$ of $Q$ is a quadratic residue modulo $4m$.
In particular, every prime divisor $p$ of $m$ has Kronecker symbol $(D/p)=+1$.
\end{lem}

\section{Dehn filling on the figure eight knot complement and its sister}

In this section we study the closed hyperbolic manifolds obtained by  Dehn filling on the figure eight knot complement and its sister,  denoted $m004$ and $m003$ respectively in SnapPea notation. We will use the ideas described in the previous section to prove the following.

\begin{thm} \label{closed_unique}
 There is an infinite sequence of hyperbolic Dehn fillings $M(a_i,b_i)$
on the figure eight knot complement $M=m004$ with $a_i^2+12b_i^2 \to  \infty$
such that if $N$ is any orientable hyperbolic 3-manifold with 
$\vol(N)=\vol(M(a_i,b_i))$ then $N$ is homeomorphic to $M(a_i,b_i)$.
So the manifolds $M(a_i,b_i)$ are determined by their volumes, amongst all 
finite volume orientable hyperbolic 3-manifolds.
\end{thm}

For $m003$ and $m004$, the quadratic forms defined in  equation (\ref{eq1}) are closely related to the quadratic form
\begin{equation}  \label{eq10}Q_0(a,b) = a^2+ab+b^2 = |a+b \omega|^2
 \end{equation} 
where $\omega = {1\over 2}(1-\sqrt{3} i )$ is a cube root of $-1$ satisfying
$\omega^2-\omega+1=0$ and $\bar \omega = 1-\omega$.
 From the symmetries of the lattice $\Z+\Z\omega$, we have
$$|a+b \omega| = | \omega^k(a+b \omega)| = | \omega^k(a+ b \bar \omega)| = | \omega^k(a+b-b \omega)|$$
for $k =0, 1, \ldots, 5$. Hence $Q_0$ has symmetry group $G_{Q_0} \cong D_6$ of order 12, and
the orbit of $(a,b)$ under $G_{Q_0}$ is
\begin{equation}  \label{eq11}\{ \pm(a,b), \pm(-b,a+b), \pm(-a-b,a), \pm (a+b,-b),\pm(b,a), \pm(-a,a+b) \}.
 \end{equation}

\medskip
Now consider a maximal horospherical cusp torus for $m004$ and $m003$.
In each case we choose the geometric basis for the peripheral homology used by SnapPea,
consisting of two shortest simple closed geodesics.
From Snap or SnapPy (or \cite{weeks_phd}) we have the following:

\smallskip
For $m004$, this torus has shortest closed geodesic of length 1,
cusp shape $\tau = -2 \sqrt{3}i = 4 \omega -2$ and area $2 \sqrt{3}$. 
Hence, the $(a,b)$ geodesic has length squared
\begin{equation}  \label{eq12}Q_1(a,b)= |a + b \tau|^2 = a^2 + 12 b^2 =Q_0(a-2b,4b) .
 \end{equation}
This quadratic form has symmetry group
$G_{Q_1} = D_2$, given by $(a,b) \mapsto (\pm a, \pm b)$,
and these all extend to symmetries of $m004$, so $G_{m004} = G_{Q_1}$.

\smallskip
For $m003$, the torus has shortest closed geodesic of length 2,
cusp shape $\tau = \omega$ and area $2 \sqrt{3}$. 
Hence, the $(a,b)$ geodesic has length squared
\begin{equation}  \label{eq13}Q_2(a,b)= |2a + 2b \omega|^2 = 4(a^2+ab+b^2) = 4Q_0(a,b) .
 \end{equation}
This quadratic form has symmetry group $G_{Q_2}=G_{Q_0} = D_6$
as described above. However, $G_{m003} = D_2$ is the subgroup
given by $(a,b) \mapsto \pm(a,b), \pm(a+b,-b)$. (Note that
any symmetry of $m003$ must preserve the homologically trivial longitude
$(-1,2)$, hence must also preserve the orthogonal $(1,0)$ curve.)

\newpage
\begin{prop}  \label{p_with_big_gap}
Given any integer $g >0$ there exist infinitely many primes $p$ such that
\begin{enumerate}
\item[(i)] 
$p$ can be written as $Q_1(a,b)$ where $a,b$ are relatively prime integers,
and if $p=Q_1(a,b)=Q_1(a',b')$ then $(a',b') = (\pm a, \pm b)$ (with 4 possible choices of signs),
\item[(ii)] $p+k$ has no primitive integer representation as $Q_0(a,b)$ for $0<|k|\le g$,
\item[(iii)] $p$ has no primitive integer representation as $Q_2(a,b)$.
\end{enumerate}
\end{prop}

This follows from some elementary number theory.

\begin{lem} \label{p=1mod12} If $p$ is a prime congruent to $1$ modulo $12$, 
then $p$ can be written uniquely in the form
     $p= Q_1(a,b) = a^2 + 12 b^2$
where $a,b$ are relatively prime positive integers.
\end{lem}

\begin{proof}
Euler showed that each prime $p \equiv 1 \bmod 3$ can be written uniquely in the form
$p = x^2 + 3 y^2$ with $x,y$ positive integers (see for example \cite[Chapter 1]{cox}).
If $p \equiv 1 \bmod 12$, then taking congruences $\bmod 4$ shows that $y$ is even.
\end{proof}

\begin{lem} \label{p=5mod6} If $p$ is a prime congruent to  $5$ modulo $6$ and $p$ divides $n$,
then $n$ has no representation
$n = Q_0(a,b)=a^2 + ab + b^2$ with $a,b$ relatively prime integers.
\end{lem}

\begin{proof} Let $p$ be an odd prime dividing $n = a^2 + ab + b^2$
with $a,b$ relatively prime integers.
Then by Lemma \ref{discr_quad_res},
$-3$ is a quadratic residue mod $p$. If $p \ne 3$, it follows that
 $p \equiv 1 \bmod 6$ by Quadratic Reciprocity 
(see e.g. \cite[Thm 96]{HaWr}). 
\end{proof}

\begin{proof}[Proof of Proposition \ref{p_with_big_gap}]
We use the previous lemmas together with the Chinese remainder theorem and Dirichlet's theorem on primes in arithmetic progressions. 

For $i=0,1,2$ let  
    $$S_i = \{ Q_i(a,b) : a,b \text{ are relatively prime integers} \}.$$
    
Given any integer $g>0$,
let $p_1, p_2, \ldots, p_{2g}$ be distinct primes congruent to $5$ modulo $6$.
Consider the integers $n$ satisfying the congruences
\begin{eqnarray*}
 n-i &\equiv& 0 \bmod p_i   \qquad\text{ for }1 \le i \le g,\\
  n+i &\equiv& 0 \bmod p_{g+i} \quad\text{ for }1 \le i \le g, \text{ and }\\
 n &\equiv& 1 \bmod 12.
 \end{eqnarray*}
By the Chinese Remainder Theorem, these have a general solution of the form
\begin{equation}  \label{eq14}
n \equiv n_0 \bmod 12 p_1 \ldots p_{2g}. 
 \end{equation}
By Dirichlet's theorem on primes in arithmetic progressions, there are infinitely
many solutions to (\ref{eq14}) such that $n$ is a prime, say $p$. For such $p$ we have
$p \equiv 1 \bmod 12$ so $p \in S_1$ by Lemma \ref{p=1mod12}. Further
$p+i$ is divisible by $p_i$ and $p-i$ is divisible by $p_{g+i}$ for $1 \le i \le g$.
Hence $p+i \notin S_0, p-i \notin S_0$ for $1\le i \le g$ by Lemma \ref{p=5mod6}.
Finally, $p \notin Q_2$ since $p$ is odd. Hence $p$ satisfies all the conditions (i)--(iii)
of the proposition.
\end{proof}

We now combine Proposition \ref{few_with_symm} and Proposition \ref{p_with_big_gap} 
to prove the main theorem of this section. 

\begin{proof}[Proof of Theorem \ref{closed_unique}]

The quadratic forms $Q_1$ and $Q_2$ for $m004$ and $m003$ considered  above
give the squares of geodesic lengths on a cusp torus with area $2 \sqrt{3}$. Hence
the corresponding normalised quadratic forms, used in Proposition \ref{error_bound}, are
$$\hat Q_1(a,b) =  {1 \over 2 \sqrt{3}} Q_1(a,b) \text{ and } \hat Q_2(a,b) =  {1 \over 2 \sqrt{3}} Q_2 (a,b).$$

We now apply Proposition \ref{p_with_big_gap} with gap size $g = 4 \sqrt{3} C$, where $C$ is the maximum of the 
constants $C_2$ for $m004$ and $m003$ given by Proposition \ref{error_bound}.  This gives 
a sequence of primes $p_i \equiv1 \bmod 12$ with $p_i \to \infty$ and pairs $(a_i,b_i)$ of relatively prime integers 
with $p_i =a_i^2+12b_i^2$ satisfying the conditions of Proposition \ref{p_with_big_gap}.

\smallskip
Then $\vol (m004(a,b)) = \vol (m004(a_i,b_i))$ implies that
 $Q_1(a,b)=Q_1(a_i,b_i)$ by Proposition \ref{few_with_symm},
 hence $(a,b) = (\pm a_i, \pm b_i)$ by condition (i) of Proposition \ref{p_with_big_gap}. 
 But these solutions differ by symmetries of $Q_1$
 that extend to symmetries of $m004$, so
 $m004(a,b)$ is homeomorphic to $m004(a_i,b_i)$.

\smallskip
Further, $\vol (m004(a_i,b_i)) = \vol (m003(a,b))$ implies that
$Q_1(a_i,b_i)=Q_2(a,b)$. But this has no solutions by condition (iii) of
Proposition \ref{p_with_big_gap}.
Hence $m004(a_i,b_i)$ is uniquely
determined by its volume amongst Dehn fillings on $m004$ and $m003$.

\smallskip
Cao and Meyerhoff \cite{CM} showed that the figure eight knot complement $m004$ and its sister $m003$
are the unique orientable cusped hyperbolic 3-manifolds of smallest volume $v_\omega$.
 The work of J{\o}rgensen and Thurston (see \cite{Th}, \cite{Gr}) shows that 
any sequence of closed hyperbolic 3-manifolds with volume approaching $v_\omega$ from below has a 
subsequence whose geometric limit is a cusped hyperbolic 3-manifold $M_\infty$ with volume $v_\omega$, 
and that the manifolds in the subsequence are obtained by Dehn filling on $M_\infty$.
It follows that there exists $\epsilon >0$ such that if $N$ is any closed orientable hyperbolic 3-manifold with 
$v_\omega - \epsilon < \vol(N) <v_\omega$ then $N$ can be obtained by Dehn filling on $m003$ or $m004$.
Now the result follows from the previous observations.
\end{proof}

\section{Dehn filling on the Whitehead link complement and its sister}

We now study  the 1-cusped hyperbolic 3-manifolds obtained 
by  Dehn filling on the Whitehead link complement and its sister,  
the complement of the $(-2,3,8)$-pretzel link, denoted 
$m129$ and $m125$ respectively in SnapPea notation.
Note that each of these manifolds has a symmetry interchanging the cusps, so it suffices
to consider Dehn fillings on cusp 0.
See  also \cite{HMW,NR} and \cite{AD} for detailed discussions
of the geometry and topology of these manifolds and their Dehn fillings.

\medskip
The main result of this section is the following:

\begin{thm}  \label{cusped_unique}
There is an infinite sequence of hyperbolic Dehn fillings $M(a_i,b_i)$
on one cusp of the $(-2,3,8)$-pretzel link complement $M=m125$ with $a_i^2+b_i^2 \to  \infty$
such that if $N$ is any orientable cusped hyperbolic 3-manifold with 
$\vol(N)=\vol(M(a_i,b_i))$ then $N$ is homeomorphic to $M(a_i,b_i)$.
So the manifolds $M(a_i,b_i)$ are determined by their volumes, amongst all 
orientable cusped hyperbolic 3-manifolds.
\end{thm}

The proof is similar to the proof of Theorem \ref{closed_unique}. However there
is an extra difficulty because not every symmetry of the cusp torus extends to a symmetry
of the manifold $m125$. To deal with this, we look at the next terms in the Neumann-Zagier asymptotic 
expansion for  volume change during Dehn filling. 

\smallskip
The quadratic forms for $m125$ and $m129$ given by (\ref{eq1}) 
are closely related to the quadratic form
\begin{equation}  \label{eq15}
Q_0(a,b) = a^2+ b^2 = |a+b i|^2 .
\end{equation} 
From the symmetries of the lattice $\Z+\Z i$, we have
$$|a+b i | = | i^k(a+b i)| = | i^k(a- b i)| $$
for $k =0, 1, \ldots, 3$,
and $Q_0$ has symmetry group $G_{Q_0} \cong D_4$ of order 8, with the orbit
of $(a,b)$ under $G_{Q_0}$ given by
\begin{equation}  \label{eq16}
\{  \pm(a, b), \pm(-b,a), \pm(a,-b), \pm(b,a) \}.
\end{equation} 

Consider a maximal horospherical cusp torus for $m129$ and $m125$.
In each case we choose the geometric basis for peripheral homology used by SnapPea.
From SnapPea (or \cite{HMW, NR, AD}) we have the following:

\smallskip
For $m125$, the torus has cusp shape $\tau = i $. If we normalise so that
the torus has area $2$ then the $(a,b)$ geodesic has length squared
\begin{equation}  \label{eq17}
Q_1(a,b)= 2 |a + b i |^2 = 2(a^2 + b^2) = 2 Q_0(a,b). 
\end{equation} 
This quadratic form has symmetry group $G_{Q_1}=G_{Q_0} = D_4$
as described above. 
The subgroup of symmetries of $Q_1$ extending to symmetries of $m125$
is the subgroup of orientation preserving symmetries $G_{m125} = C_4$, given by 
$(a,b) \mapsto  \pm(a, b), \pm(-b,a)$.

\smallskip
For $m129$, the torus has cusp shape $\tau = 2 i$. If we normalise so that
the torus has area $2$ then the $(a,b)$ geodesic has length squared
\begin{equation}  \label{eq18}
Q_2(a,b)={ |a+ 2bi|^2 } = (a^2+ 4b^2) = Q_0(a,2b).
\end{equation} 
This quadratic form has symmetry group $G_{Q_2} = D_2$
consisting of $(a,b) \mapsto (\pm a, \pm b)$, but $G_{m129} = C_2$ is
the subgroup $(a,b) \mapsto \pm(a,b)$.

\begin{prop}  \label{2p_with_big_gap}
For each $g >0$ there exist infinitely many integers $m=2p$, with $p$ prime,
such that
\begin{enumerate}
\item[(i)] $m$ can be written as $Q_1(a,b)$ where $a,b$ are relatively prime integers,
and if $Q_1(a,b)=Q_1(a',b')$ then $(a',b') = (\pm a, \pm b)$ or
$(a',b') = (\pm b, \pm a)$. 
\item[(ii)] $m+k$ has no primitive integer representation as $Q_0(a,b)$ for $0<|k|\le g$,
\item[(iii)] $m$ has no primitive integer representation as $Q_2(a,b)$.
\end{enumerate}
\end{prop}

This again follows from some elementary  number theory.
\begin{lem} \label{p=1mod4} If $p$ is a prime congruent to $1$ modulo $4$, 
then $p$ can be written  in the form
     $p= a^2 + b^2$
where $a,b$ are relatively prime positive integers.
Further, if $p=(a')^2+(b')^2$ where $a',b'$ are positive integers
then $(a',b')=(a,b)$ or $(b,a)$.
\end{lem}

\begin{proof}
This is a result of Euler 
(see, for example, \cite[Chapter 15]{HaWr}).
\end{proof}

\begin{lem} \label{p=3mod4} If  $p$ is a prime  congruent to  $3$ modulo $4$ 
and $p$ divides $n$, then $n$ has no representation
$n = Q_0(a,b)=a^2 + b^2$ with $a,b$ relatively prime integers.
\end{lem}

\begin{proof} Let $p$ be an odd prime dividing $n = a^2 + b^2$
with $a,b$ relatively prime integers.
Then by Lemma \ref{discr_quad_res},
$-1$ is a quadratic residue mod $p$. It follows that
 $p \equiv 1 \bmod 4$ 
(see e.g. \cite[Thm 82]{HaWr}). 
\end{proof}

\begin{lem} If $n \equiv 2 \mod 4$ then
$n = Q_2(a,b) = a^2 + 4 b^2 $ has no solution with $a,b$ integers.
\end{lem}
\begin{proof} Reduce modulo $4$.
\end{proof}

\begin{proof}[Proof of Proposition \ref{2p_with_big_gap}]
This follows from the previous lemmas together with the Chinese remainder theorem and Dirichlet's theorem on primes in arithmetic progressions, as in the proof of Proposition \ref{p_with_big_gap}.
\end{proof}

\begin{proof}[Proof of Theorem \ref{cusped_unique}]
The quadratic forms $Q_1$ and $Q_2$ for $m125$ and $m129$ considered  above
give the squares of geodesic lengths on a cusp torus with area $2$. Hence
the corresponding normalised quadratic forms are given by
$$\hat Q_1(a,b) =  {1 \over 2 } Q_1(a,b) \text{ and } \hat Q_2(a,b) =  {1 \over 2 } Q_2 (a,b).$$
We now apply Proposition \ref{p_with_big_gap} with gap size $g = 4 C$, where $C$ is the maximum of the 
constants $C_2$ for $m125$ and $m129$ given by Proposition \ref{error_bound}.  

This gives 
a sequence of primes $p_i \equiv 1 \bmod 4$ with $p_i \to \infty$ and pairs $(a_i,b_i)$ of relatively prime integers with
$p_i =a_i^2+b_i^2$ satisfying the conditions of Proposition \ref{2p_with_big_gap}.
Then
 $\vol(m125(a,b)) = \vol(m125(a_i,b_i)$ implies that
 $Q_1(a,b)=Q_1(a_i,b_i)$, hence   $(a,b) = \pm(a_i,b_i), \pm(-b_i,a_i), \pm(a_i,-b_i)  \text{ or} \pm(b_i,a_i)$.

\smallskip
Now the surgeries $\pm(a, b), \pm(-b,a)$ give manifolds which are homeomorphic, hence have the same volume.
Similarly, the surgeries $ \pm(a,-b),   \pm(b,a)$ give homeomorphic manifolds of the same volume.
Next, we use the degree 4 terms in the Neumann-Zagier asymptotic formula
to show that $\vol(m125(a,b)) \ne  \vol (m125(b,a))$ for all sufficiently large $(a,b)$.

\smallskip
In \cite{AD}, Aaber and Dunfield study the Neumann-Zagier asymptotic expansion 
of volume for Dehn fillings on  the Whitehead link sister $m125$.
Using their calculations we obtain  the following asymptotic formula for decrease in volume
for surgeries on cusp 0 of $m125$ (see Example \ref{m125_NZ4} and equation (\ref{eq31}) below):
\begin{equation*}  
\Delta V_{m125} (a,b)=  
{ \pi^2 \over a^2+b^2} -
\frac{\pi^4(a^4-12 a^3 b-6 a^2 b^2+12 a b^3+b^4)}{ 24\left(a^2+b^2\right)^4}
+ O \left({1 \over (a^2+b^2)^3}\right).
\end{equation*} 
Hence, we have
\begin{equation}  \label{eq20}\Delta V_{m125} (a,b) - \Delta V_{m125} (b,a)
=  \frac{\pi ^4 a b \left(a^2-b^2\right)}{\left(a^2+b^2\right)^4} +
O \left({1 \over (a^2+b^2)^3}\right).
\end{equation} 

\begin{lem} \label{lower_bound}
If $a$ and $b$ are integers with $a > b > 0$,
then
$a b (a^2-b^2) \ge  {r^3/4},$
where $r^2 = a^2+b^2$.
\end{lem}

\begin{proof} Consider the function 
$$f(a,b)=    a b (a^2-b^2) = a b ( a + b ) ( a - b ).$$
Now we have
\begin{enumerate}
\item[(i)]  If $b \le a/2$  then   
    $b(a-b) \ge b (a - a/2)=  b a/2  \ge a/2,$
since $b \ge 1$.
\item[(ii)] If $b  \ge a/2$ then
   $b (a-b) \ge  a/2 \cdot (a-b)  \ge  a/2,$
since $a-b  \ge 1$.
\end{enumerate}
\smallskip
Further, $a^2+b^2 = r^2$ and $ a > b > 0$ imply that
$a^2 \ge r^2/2$ and $a+b \ge r$. Hence
$$f(a,b)=  a(a+b) b(a-b) \ge a(a+b) \cdot a/2  = a^2 (a+b)/2 \ge r^3/4,$$
as desired. 
\end{proof}

Equation (\ref{eq20}) shows that,  whenever $a^2+b^2$ is sufficiently large,
$$\Delta V_{m125} (a,b) - \Delta V_{m125} (b,a)
\ge  \frac{\pi ^4 a b \left(a^2-b^2\right) - C(a^2+b^2)}{\left(a^2+b^2\right)^4},$$
where $C$ is a positive constant.
Hence, by  Lemma \ref{lower_bound}, 
there exists a constant $R>0$ such that 
if $(a,b)$ are relatively prime integers with $a >b >0$ and 
$a^2+b^2 \ge R^2$ , then 
$\Delta V_{m125} (a,b) > \Delta V_{m125} (b,a)$.

\smallskip
This completes the proof that for all sufficiently large $(a_i,b_i)$,
the Dehn fillings $m125(a_i,b_i)$ are determined by their volumes
amongst Dehn fillings on $m125$.

\smallskip
Further, $\vol(m125(a_i,b_i)) = \vol(m129(a,b))$ implies that $Q_1(a_i,b_i)=Q_2(a,b)$,
which has no solutions by condition (iii) of Proposition \ref{2p_with_big_gap}. Hence,
for all sufficiently large $(a_i,b_i)$,
the manifolds $m125(a_i,b_i)$ are determined by their volumes amongst Dehn fillings on $m129$ and $m125$.

\smallskip
The result now follows as in the proof of Theorem \ref{closed_unique} since
$m129$ and $m125$ are the unique 2-cusped manifolds of volume $v_{\omega^2}$
by the work of Agol \cite{agol_2cusp}. So any 1-cusped manifold with volume less than $v_{\omega^2}$
but sufficiently close to $v_{\omega^2}$ comes from Dehn filling on $m129$ or $m125$.
\end{proof}

\section{Higher order terms in the Neumann-Zagier asymptotic formula}\label{degree4_NZ} 

Let $M$ be a cusped hyperbolic 3-manifold with horospherical torus cross section $T$,
and chose meridian and longitude generators for $H_1(T)=\Z^2$ as in section \ref{cusped_hyp}.
Let $u,v$ denote the logarithms of the holonomies of the 
meridian and longitude respectively, chosen so that $u=v=0$ at the complete hyperbolic structure on $M$.
 Then $-v$ has a power series expansion
 \begin{equation}  \label{eq21}
 -v = \sum_{n=1}^\infty   c_n u^n ,
\end{equation} 
 where $c_n=0$ for all even $n$ and $c_1 = -\tau$ where $\tau$ is our cusp shape 
 parameter.\footnote{In \cite{NZ},  a non-standard orientation convention for meridian
and longitude was used; this was pointed out in \cite{NR}. 
Replacing the basis in \cite{NZ} by $meridian, -(longitude)$ corresponding to $u,-v$ 
gives our version of the Neumann-Zagier formula.}

Then  Neumann and Zagier \cite[equation (62), p328]{NZ}) derive the following asymptotic formula
for the decrease in volume $\Delta V$ during Dehn filling for $(p,q)$ Dehn filling
on one cusp $T$ of the cusped hyperbolic 3-manifold $M$:
\begin{equation}  \label{eq22}
\Delta V = {|\text{Im}\,(c_1)| \pi^2 \over |z|^2} -
{2 \pi^4}\,\text{Im}\, \left[  {c_3  \over z^4}\right] + O \left({1 \over |z|^6}\right),
\end{equation} 
where  $z =  p + q \tau$.

If we write $z = r e^{i \theta}$, then this becomes
\begin{equation}  \label{eq23}\Delta V = {|\text{Im}\,(\tau)| \pi^2 \over r^2} -
{2 \pi^4 \over r^4 }\text{Im}\, \left[  c_3 e^{-4 i\theta}\right] + O \left({1 \over r^6}\right).
\end{equation} 

We now give some calculations for the manifolds studied in the previous sections.

\begin{ex}\label{m004_NZ4}
{\em 
For $m004$, Neumann and Zagier (\cite[p331]{NZ}) show that
$c_1 = 2 \sqrt{-3}=-\tau$ and $c_3 = {2 \sqrt{-3} \over 3}$,
and obtain the following asymptotic formula:
\begin{equation}  \label{eq24}\Delta V_{m004}(p,q) = {2 \sqrt{3} \pi^2 \over p^2+12q^2}
-{4\sqrt{3}(p^4-72p^2q^2+144q^4)\pi^4 \over 3(p^2 + 12 q^2)^4} + O \left({1 \over (p^2 + 12 q^2)^3}\right).
\end{equation} 
Writing $z = p + \tau q = r e^{i \theta}$ as above, this simplifies to
\begin{equation}  \label{eq25}
\Delta V_{m004} = {2 \sqrt{3} \pi^2 \over r^2} -
{4 \pi^4 \cos(4\theta) \over \sqrt{3} r^4 }+ O \left({1 \over r^6}\right).
\end{equation} 
}
\end{ex}

\begin{ex} \label{m003_NZ4}
{\em 
 Let $WL$ denote the Whitehead link complement 
as drawn in Figure \ref{fig.WL_WS} (and in \cite{HMW}, \cite{NR}),
with the standard topological choice of meridians and longitudes.
Then  Hodgson-Meyerhoff-Weeks  \cite[ Figure 8 with $p=q=1$]{HMW} show that 
the hyperbolic Dehn fillings $WL(1,1)(m,l)$ and $WL(-5,1)(m,-l-m/2)$
have equal volumes since they have homeomorphic 2-fold covers,
whenever $m,l$ are relatively prime integers with $m$ a multiple of $4$.

Now $WL(1,1)$ and $WL(-5,1)$ are homeomorphic to $m004$ and $m003$
respectively. In fact, using SnapPy or Snap,  there are isometries taking
$WL(1,1)(m,l)$ to $m004(m,l)$
and $WL(-5,1)(m,-l-m/2)$ to $m003(m/2-l,2l)$.
Hence we have
\begin{equation}  \label{eq26}\vol(m004(m,l)) =  \vol(m003(a,b))  
\end{equation} 
for hyperbolic Dehn fillings such that $(a,b)=(m/2-l,2l)$ and $(m,l)=(2a+b,b/2)$ for 
pairs of relatively prime integers 
with $m$ a multiple of 4. 
 
 \smallskip
 Since the volume is a real analytic function on hyperbolic Dehn surgery space
(using $u$ as coordinate, as in \cite{NZ}),
it follows that (\ref{eq26}) holds for all large real values of $(p,q)$. Hence we can
deduce the asymptotic formula for decrease in volume for $m003$:
\begin{align}\label{eq27}
\Delta V_{m003}(a,b) &= \Delta V_{m004}(2a+b,b/2) \notag\\
&=  \frac{\sqrt{3} \pi ^2}{2 \left(a^2+a b+b^2\right)}
-\frac{\pi ^4 \left(-18 b^2 (2 a+b)^2+(2 a+b)^4+9 b^4\right)}{64 \sqrt{3} \left(a^2+a
   b+b^2\right)^4}
 + \ldots
 \end{align}
Introducing polar coordinates 
$r e^{i \theta} =  a+ b \omega$,
this simplifies to
\begin{equation}  \label{eq28}
\Delta V_{m003}(a,b) =   { \sqrt{3} \pi^2 \over 2 r^2} - { \pi^4 \over 4 \sqrt{3}}{ \cos(4\theta) \over r^4} + 
O\left({1 \over r^6}\right).
\end{equation} 
}
\end{ex}

\begin{ex}\label{m125_NZ4}
{\em
For the Whitehead link sister $m125$, Aaber and Dunfield \cite{AD}
study the Neumann-Zagier asymptotic expansion of volume for Dehn fillings.
If  $u,v$ denote the  logarithms of holonomies of the 
 meridian and longitude for the geometric
 basis for cusp 0 of $m125$, 
 then they show \cite[p1019]{AD} that
\begin{equation}  \label{eq29}
-v = i u + { -3 + i \over 48} u^3 +  \ldots
\end{equation} 
 so 
$c_1 = -\tau = i$ and $c_3 = {-3 + i \over 48}$ in our notation.
This gives the following asymptotic formula for decrease in volume
for surgeries on cusp 0:
\begin{equation}  \label{eq30}
\Delta V_{m125} (p,q)=  
{ \pi^2 \over |z|^2} -
{2 \pi^4}\,\text{Im}\, \left[  {-3 + i \over 48 z^4}\right] + O \left({1 \over |z|^6}\right),
\end{equation} 
where  $z =  p + q \tau = p - q i$.

More explicitly: 
\begin{equation}  \label{eq31}
\Delta V_{m125} (p,q)=  
{ \pi^2 \over p^2+q^2} -
\frac{\pi^4(p^4-12 p^3 q-6 p^2 q^2+12 p q^3+q^4)}{ 24\left(p^2+q^2\right)^4}
+ O \left({1 \over (p^2+q^2)^3}\right).
\end{equation} 
}
\end{ex}

\begin{ex} \label{m129_NZ4}
{\em For the Whitehead link complement $WL$ (as drawn in Figure \ref{fig.WL_WS})
we can compute the asymptotic expansion for $\Delta V$ using the
work of Neumann-Reid \cite{NR}.  Let  $u,v$ denote the  logarithms of holonomies of the 
 standard meridian and longitude for the link complement. Then \cite{NR} shows that
 $$u = \log x + \log(x+1) - \log(x-1) \text{ and } v = 4 \log x - 2\pi i,$$
 where $x \in \C$ is a suitable simplex parameter. From these equations we
 obtain  $$e^u = { x(x+1) \over (x-1)} \text{ where } x = i e^{-v/4}.$$
 This gives a quadratic equation for $x$ with the relevant solution given by
$$x=\frac{1}{2} \left(\sqrt{-6 e^u+e^{2 u}+1}+e^u-1\right) = i e^{-v/4},$$
and so
\begin{align} \label{eq32}
-v &= -4 \log\left(\frac{-i}{2} \left(\sqrt{-6 e^u+e^{2 u}+1}+e^u-1\right)\right) \notag\\
&= (-2+2i) u + {i u^3 \over 6} +  O(|u|^5) .
\end{align}
Hence $c_1 = -\tau = (-2+2i)$ and $c_3= { i \over 6}$ in our previous notation.

\smallskip
This gives the following asymptotic formula for decrease in volume
for Dehn filling on one cusp of $WL$:
\begin{equation}  \label{eq33}
\Delta V_{WL} (p,q)=  
{ 2 \pi^2 \over |z|^2} -
{2 \pi^4}\,\text{Im}\, \left[  { i \over 6 z^4}\right] + O \left({1 \over |z|^6}\right),
\end{equation}
where  $z =  p + q \tau = p + q (2-2i)$.
Hence
\begin{equation}  \label{eq34}
\Delta V_{WL}(p,q)=
\frac{2 \pi ^2}{p^2+4 p q+8 q^2}-\frac{\pi ^4 \left(p^2-8 q^2\right) \left(p^2+8 p q+8 q^2\right)}{3 \left(p^2+4 p q+8
   q^2\right)^4} + \ldots.
\end{equation}

We can convert from the standard peripheral curves for $WL$
 to the geometric peripheral curves on $m129$ using:
$WL(p,q)=m129(a,b)$ where $a=p+2q,  b=-q$ or $p=a+2b,q = -b$.
Then the new cusp shape parameter for  $m129$ is $\tau' = -2 i$ and
\begin{align}\label{eq35}
\Delta V_{m129}(a,b) &= \Delta V_{WL}(a + 2 b, -b) \notag\\
&= \frac{2 \pi ^2}{a^2+4 b^2} -\frac{\pi ^4 \left(a^4-24 a^2 b^2+16 b^4\right)}{3 \left(a^2+4 b^2\right)^4}
+ \ldots
\end{align}

}
\end{ex}

\section{Hyperbolic graphs with parabolic meridians}\label{sec.Vol}
In this section, we recall some results on hyperbolic graphs in $S^3$ with parabolic meridians; 
for details see \cite{HHMP} and \cite{Mas}.
In this paper, spatial graphs may contain link components, and we regard every link as a spatial graph without vertices.
We use the word {\em link} to emphasize that the graph does not have any vertices.
Let $G$ be a spatial trivalent graph in $S^3$ and let $N$ be a manifold obtained from $S^3\setminus G$ 
by removing an open regular neighbourhood of each vertex.
Then $N$ is a manifold with boundary consisting of thrice punctured spheres, one corresponding to each vertex of $G$. We say that
$G$ is {\em hyperbolic} if $N$ admits a hyperbolic metric of finite volume with geodesic boundary;
then the meridians of the graph correspond to parabolic isometries.
We can construct a hyperbolic link from a hyperbolic graph by using the following.

\begin{lem}[\cite{Mas}]\label{Vol}
For hyperbolic graphs, the moves shown in
Figure \ref{fig.move1}  are volume preserving, where regular neighbourhoods of the
two trivalent vertices have been removed in middle diagram.
\end{lem}
\begin{figure}[h]
\includegraphics[scale=1]{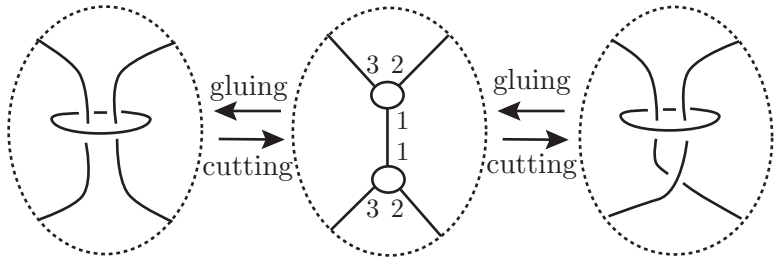}
\caption{Volume preserving moves on hyperbolic graphs}
\label{fig.move1}
\end{figure}
\begin{proof}
Note that each vertex in a hyperbolic graph corresponds to a totally geodesic thrice punctured sphere and each edge corresponds to an annulus cusp.
Moreover, the hyperbolic structure on a thrice punctured sphere is unique by \cite{Ada};
hence any orientation preserving homeomorphism between hyperbolic thrice punctured spheres is isotopic to an isometry.
Since any homeomorphism between thrice punctured spheres is uniquely determined by its action on the cusps, we may denote homeomorphisms as elements of $S_3$,
the symmetry group of degree 3.

\smallskip
Fix labels for the cusps of the thrice punctured spheres as in Figure \ref{fig.move1}.
If we glue the boundaries via the homeomorphism corresponding to the identity permutation,
then we get the tangle on the left of Figure \ref{fig.move1}.
If we glue the boundaries via the homeomorphism corresponding to the permutation
interchanging 2 and 3,
we get the tangle on the right of Figure \ref{fig.move1}.

\begin{figure}[h]
\includegraphics[scale=0.9]{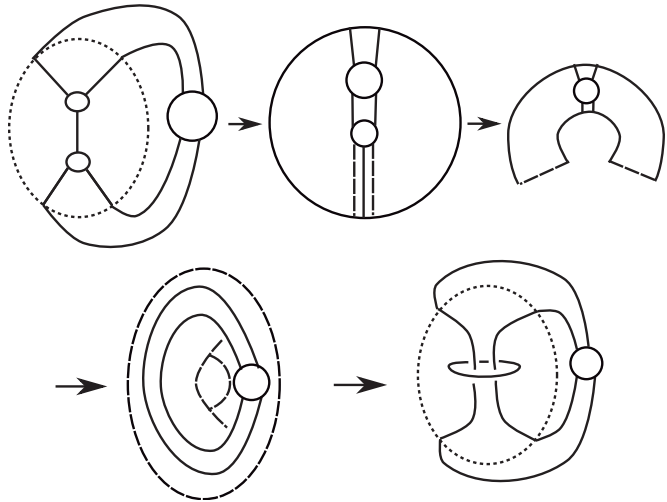}
\caption{Gluing by a homeomorphism}
\label{fig.seq}
\end{figure}

This is illustrated in Figure \ref{fig.seq}. Here we regard $S^3$ as $B^3\cup B^3$ where one ball $B^3$ is a
regular neighbourhood of a vertex of the graph shown in the middle of Figure \ref{fig.move1}.
We then reverse the inside and the outside with respect to the boundary sphere $S^2 = \bd B^3$.
After gluing, we get a link or graph in $S^2\times S^1$.
The component coming from the cusp labelled by $1$ is a loop which corresponds to a generator of the fundamental group of $S^2\times S^1$, and its complement is homeomorphic to a solid torus.
Thus we get a tangle as shown on the left or right of Figure \ref{fig.move1}.

\smallskip
The inverse move (left or right to middle) is valid since any essential thrice punctured sphere (or, $2$-punctured disk) in a hyperbolic 3-manifold is totally geodesic (\cite{Ada}).
Since we are dealing with hyperbolic graphs, the $2$-punctured disk on the left or right of Figure \ref{fig.move1} is essential (see \cite{Pur} Lemma 2.1).
\end{proof}

\begin{defi}
{\em
A {\em fully augmented link} is a link obtained by applying the gluing moves of Lemma \ref{Vol} to a planar diagram of a hyperbolic graph (embedded in the plane) 
so that the resulting graph is actually a link.
A {\em crossing circle} is a component of a fully augmented link which appears as a circle in Figure \ref{fig.move1}.
A {\em knot component} of a fully augmented link is a component which is not a crossing circle.
}
\end{defi}
\begin{rmk}
{\em These definitions of a fully augmented link, a crossing circle, and a knot component are equivalent to the original definitions in \cite{FP} or \cite{Pur}.}
\end{rmk}

\subsection{Polyhedral decomposition for the complement of  a hyperbolic graph with a planar diagram}
\label{polyhedral_decomp}

Given a hyperbolic graph $P$ with a planar diagram $D$, let $V$ be the set of vertices of $P$ and 
$N_P = S^3\setminus P \setminus (\cup_{v\in V} \mathcal{N}(v) )$ where $\mathcal{N}(v)$ is a regular neighbourhood of $v$. 
Let $\Pi$ be the plane containing $D$. 
Then $N_P$ admits hyperbolic metric of finite volume with geodesic boundary.
We can decompose $N_P$ into two isometric convex ideal hyperbolic polyhedra by the following procedure, depicted in Figure \ref{fig.ATdecomp}. 
\begin{description}
\item[Step 1] Cut $N_P$ along $\Pi$. 
\item[Step 2] Collapse each edge of $P$ to a point.
\end{description}
Let $L$ be a fully augmented link which is obtained from $D$ by the gluing moves of Lemma \ref{Vol}.
Then this decomposition is exactly equal to the decomposition of $S^3\setminus L$ found by 
Agol-Thurston in \cite{Lac} (see also \cite{FP}, \cite{Pur}).

\begin{figure}[h]
\includegraphics[scale=1]{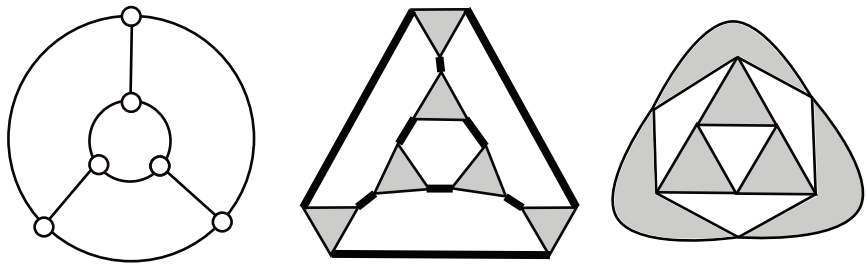}
\caption{Polyhedral decomposition of a graph complement}
\label{fig.ATdecomp}
\end{figure}

We shade the faces coming from the thrice punctured sphere boundaries of $N_P$ and leave the faces on $\Pi$ white as in \cite{FP}, \cite{Pur}.
Note that every shaded face is an ideal triangle, and since reflection $\phi$ about $\Pi$ leaves each thrice punctured sphere boundary invariant, 
the shaded faces are all orthogonal to the adjacent white faces.
Hence the dihedral angle of each edge is $\pi /2$.
Since $\phi$ preserves white faces, each face of these polyhedra is totally geodesic.
Hence each white face extends to the boundary at infinity $S^2_{\infty}$ of hyperbolic space to give a Euclidean circle and 
such circles corresponding to adjacent faces are tangent to each other.
Therefore the white faces of each polyhedron determine a circle packing of $S^2_{\infty}$ whose nerve is isomorphic to the dual of the original graph diagram $D$.

\section{Associated circle packings and cusp moduli}\label{pack}
In this section we study the circle packing corresponding to the hyperbolic graph $W_n$ shown in Figure \ref{fig.graphW}. 

\begin{figure}[h]
\includegraphics[scale=1]{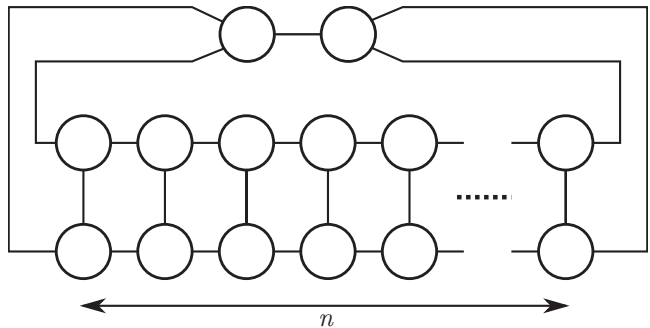}
\caption{Graph $W_n$ with $2n+2$ vertices}
\label{fig.graphW}
\end{figure}

The two isometric polyhedra in the decomposition of the complement of $W_n$ described in section \ref{sec.Vol} are obtained by
gluing together $n$ regular ideal octahedra (see Figure \ref{fig.graphW2} (right)).
Therefore the volume of its complement is $2nV_8$ where $V_8= 3.663862 \ldots$ is the volume of the regular ideal octahedron.
A circle packing corresponding to the graph is depicted in Figure \ref{fig.graphW2} (left). 
Each tangency point of the circle packing corresponds to an annulus cusp of the complement of $W_n$.
\begin{figure}[h]
\includegraphics[scale=1]{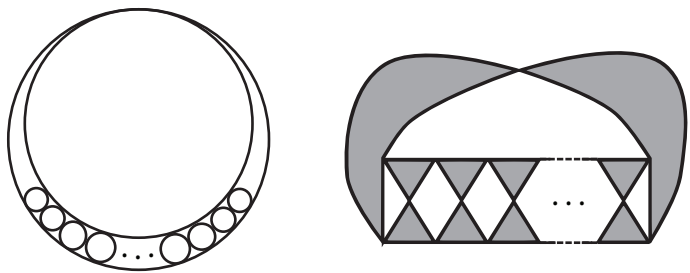}
\caption{A circle packing obtained from the graph $W_n$ (left) and a (topological) polyhedral decomposition (right)}
\label{fig.graphW2}
\end{figure}

Recall from Remark \ref{cusp_param} that each torus cusp of a hyperbolic $3$-manifold has a 
Euclidean structure $\C/\Gamma$
where $\Gamma \cong \Z^2 \subset \C$ is a lattice,  and its shape parameter is determined by giving
 two independent translations corresponding to a choice of basis for $\Gamma$.
Since there are infinitely many ways to choose such a basis, this shape parameter is not unique.
The {\em modulus} of a torus cusp is defined as the complex ratio
\begin{equation} \label{eq_mod}
z = (\text{second shortest translation})/(\text{shortest translation}).
\end{equation}
It is a complex number lying in the region $|$Re$(z)| \leq 1/2$  and  $|z|\geq 1$.
If $z$ lies on the boundary of this region, we choose it so that Re$(z) \geq 0$.

\smallskip
For the complement of $W_n$, there are three different types of annulus cusps, 
see Figures \ref{fig.pack1}, \ref{fig.pack2} and \ref{fig.pack3}.
The circle packing diagrams giving the cusp shapes are obtained by M\"{o}bius transformations 
which map one of the marked cusps to infinity.
Since the graph complement is reconstructed by gluing the two isometric polyhedra along their corresponding white faces, 
we get the boundary of a horoball neighbourhood of the annulus cusp
by gluing two copies of rectangles together along corresponding edges depicted by solid lines.
Later, we glue these annuli cusps together along their boundaries (depicted as dotted lines)  to get a torus cusp, and compute its modulus.

\smallskip
By using Lemma \ref{Vol}, we obtain hyperbolic links from $W_n$.
The modulus of each cusp can be computed by knowing how to glue the rectangles according to the gluing pattern of thrice punctured spheres.
\begin{figure}[h]
\includegraphics[scale=0.9]{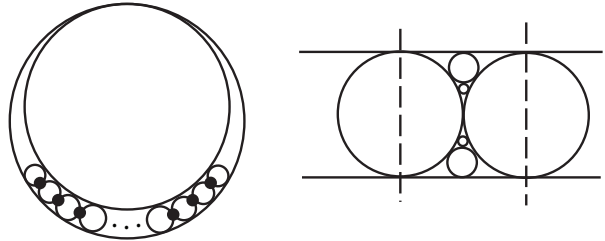}
\caption{Cusp modulus for ideal vertices of type 1}
\label{fig.pack1}
\end{figure}
\begin{figure}[h]
\includegraphics[scale=0.9]{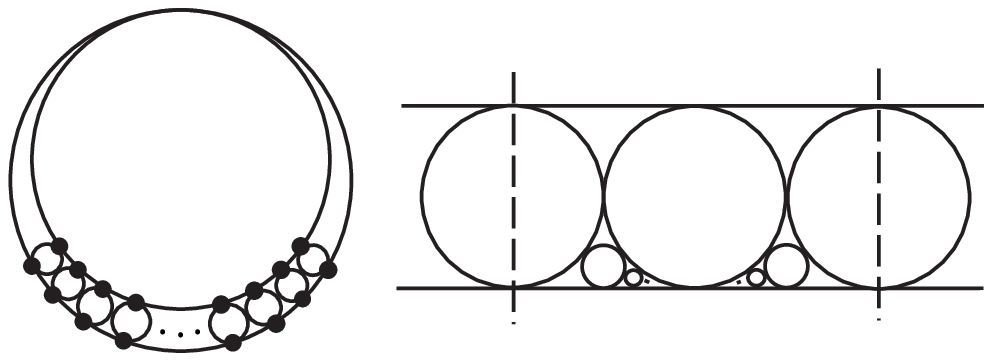}
\caption{Cusp modulus for ideal vertices of type 2}
\label{fig.pack2}
\end{figure}
\begin{figure}[h]
\includegraphics[scale=0.9]{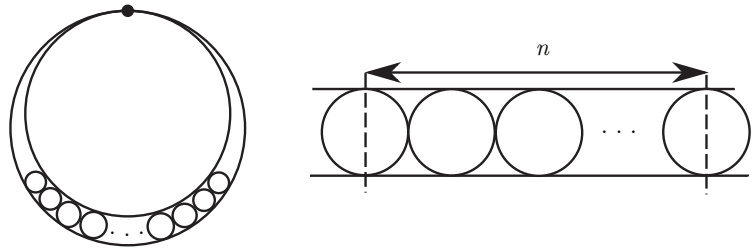}
\caption{Cusp modulus for ideal vertices of type 3}
\label{fig.pack3}
\end{figure}

\begin{ex}\label{example}
{\em
The graph $G_{n,k}$ 
in Figure \ref{fig.example} (top) is obtained from the graph in Figure \ref{fig.graphW} by $k+2$ 
applications of Lemma \ref{Vol}, where $k+2 \le n$.
The torus cusp corresponding to the knot component (with $k$ crossings) 
comes  from the $2(k+1)$ marked vertices 
in Figure \ref{fig.example} (bottom).
Note that in Figure \ref{fig.example} (also in Figure \ref{fig.source}), 
the choice of crossings below the crossing circles is not important, because changing 
an over crossing to an under crossing does not change the graph complement.

\begin{figure}[h]
\includegraphics[scale=1]{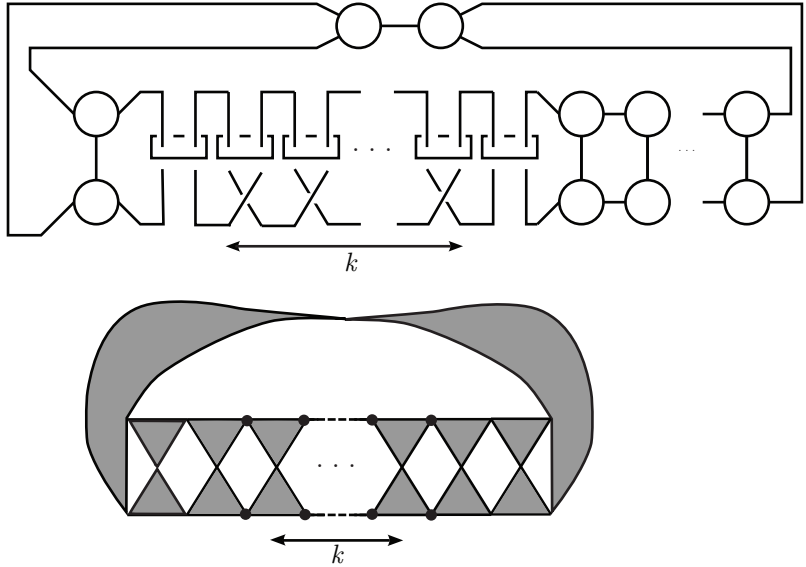}
\caption{A graph $G_{n,k}$ obtained from $W_n$ (top) and 
vertices corresponding to the knot component cusp (bottom)
}
\label{fig.example}
\end{figure}

In Figure \ref{fig.cuspshapes}, we depict the shapes of the torus cusps.
The left diagram shows the shape of a crossing circle cusp without a half twist below,
the middle shows  the shape of a crossing circle cusp with a half twist,
and the right shows the shape of the cusp which comes from the knot component.

\begin{figure}[h]
\includegraphics[scale=1]{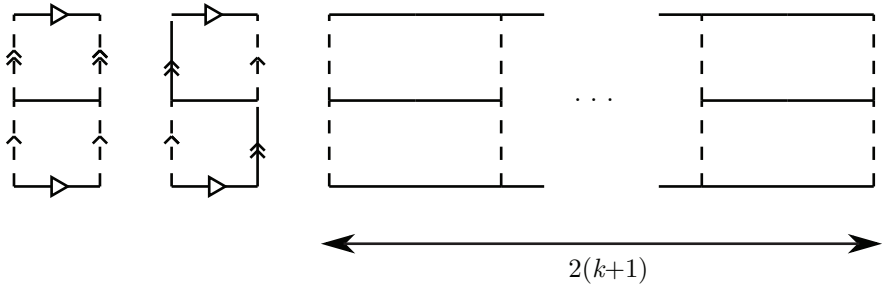}
\caption{Shapes of the torus cusps for the graph $G_{n,k}$ in Figure \ref{fig.example}}
\label{fig.cuspshapes}
\end{figure}

The shape of the torus cusp of the knot component in Figure \ref{fig.example} is obtained by arranging 
$2 \times 2(k+1)$ copies of the $1 \times 2$ 
rectangle in Figure \ref{fig.pack2} to form a
$2 \times 4(k+1)$ rectangle with opposite edges identified by translations.
Therefore, its cusp modulus is $2(k+1)\sqrt{-1}$.
The modulus of each crossing circle cusp depends on the existence of a half twist;
if it has a half twist below, the modulus is $\sqrt{-1}$ and if not, it is $2\sqrt{-1}$. 
(Note that these moduli do not depend on the total number of vertices in the graph $G_{n,k}$.)
}
\end{ex}

\section{Construction of different manifolds sharing the same volume}\label{const}
We start with the graph with $2n$ vertices shown in Figure \ref{fig.source}.
This graph is obtained by applying the gluing move in Lemma \ref{Vol} $n+1$ times to $W_{2n}$.

\begin{figure}[h]
\includegraphics[scale=1]{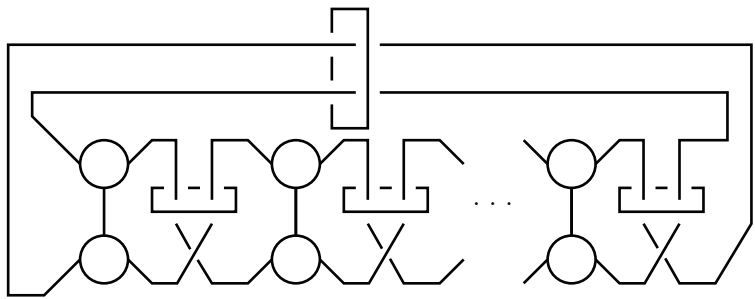}
\caption{Starting point: graph with $2n$ vertices (represented by circles)}
\label{fig.source}
\end{figure}

First, we label every remaining pair of thrice punctured spheres by an element of $\{0,1\}$.
For each such pair 
we apply Lemma \ref{Vol} as follows. 
We glue the pair of thrice punctured spheres labelled by $0$ (resp. $1$) without (resp. with) a half twist, as in Figure \ref{fig.move1} (left) (resp.(right)).
By this construction, we get a hyperbolic link complement $M_b$ for each cyclic binary word $b$ of length $n$ and its volume is $4n V_8$.

\smallskip
Now we consider the decomposition of $M_b$ into isometric convex hyperbolic polyhedra with shaded faces described in section 
\ref{polyhedral_decomp}. Note that every edge lies in exactly one shaded triangular face,
since the faces have a checkerboard colouring.

\begin{defi}[\cite{FP}, \cite{Pur}]{\em
Given an edge $e$ in such a polyhedral decomposition, 
let $t$ be the shaded
ideal triangle containing $e$. Then  the {\em midpoint} $m$ of $e$ is the
foot of the perpendicular in $t$ from the vertex of $t$ opposite $e$ to $e$.
}
\end{defi}
Since isometries preserve angles, when we glue two thrice punctured sphere boundaries together by an isometry, it maps midpoints to midpoints.
Moreover, for fully augmented links, we may expand horoball neighbourhoods of the cusps to find a horoball packing such that the horoballs bump together precisely at the midpoints of edges. 
More precisely,
\begin{thm}[\cite{FP}, \cite{Pur}]\label{thm.packing}
Let $L$ be a fully augmented link. Then there exist horoball neighbourhoods of the cusps of $S^3\setminus L$ such that the midpoint of every edge is a point of tangency of the corresponding horospherical tori.
\end{thm}

\begin{rmk}\label{horo_packing_invariant}
{\em
For each $M_b$, we fix the horoball packing $\mathcal{H}_b$  given by Theorem \ref{thm.packing}.
In fact, this horoball packing depends only on the geometry of $M_b$ and not on any choice of
polyhedral decomposition for $M_b$. This follows since the horospherical tori for $\mathcal{H}_b$
 can be defined as follows:  For each cusp with modulus $\sqrt{-1}$ or $2\sqrt{-1}$ we choose
 the horospherical torus of area $2$, and for each cusp with modulus $m \sqrt{-1}$ with $m>2$
 we choose the horospherical torus of area $4m$.
 These areas are easily  calculated from the diagrams in 
 Figures \ref{fig.pack1}--\ref{fig.pack3} and Figure \ref{fig.cuspshapes}, 
 using the fact that each dotted edge  in Figures 10-12 has length 1 in the horospherical tori for $\mathcal{H}_b$,
 (since it is a horocycle joining midpoints of edges in an ideal hyperbolic triangle).
 }
 \end{rmk}

\begin{thm}\label{change_mfd}
Let $b_i$ be a binary cyclic word of length $n\ge 3$ 
and let $M_{b_i}$ be the hyperbolic link complement
obtained from $b_i$ by the above procedure for $i = 1,2$.
Then $M_{b_1}$ is isometric to $M_{b_2}$ if and only if $b_1$ and $b_2$ are related by the natural action of $D_n$, the dihedral group of order $2n$.
\end{thm}
\begin{proof}
Let $b$ be a binary cyclic word of length $n$. If $b$ contains $k \ge 1$ zeros,
then we can split $b$ into $k$ subwords of the  form $p_i =  011 \cdots  10$
with $i$ $1$'s $(i\geq 0)$ by using each zero
in $b$ exactly twice.

After the gluing move in Lemma \ref{Vol}, each $p_{i}$ produces a torus cusp with $2i+1$ crossings, whose modulus is $m_i = 2(2i+2)\sqrt{-1} = 4(i+1)\sqrt{-1}$ (see Example \ref{example}).
Let $\{ p_{i_1}, p_{i_2},..., p_{i_k} \}$ be the decomposition of $b$, and let 
$C_j$ be the cusp corresponding to the knot component produced by $p_{i_j}$. 
We get a cyclic sequence $\{m_{i_1}, m_{i_2}, ..., m_{i_k}\}$ of cusp moduli from $b$.
Note that the total number of cusps of $M_b$ is $2n+1+k$;  the crossing circle cusp 
$C$ at the top of
Figure \ref{fig.source} has modulus $n \sqrt{-1}$ (from the cusp diagram in Figure \ref{fig.pack3}), 
and the other $2n$ crossing circle cusps
have modulus $\sqrt{-1}$ or $2\sqrt{-1}$ by Example \ref{example}.

We now only consider cusps whose modulus has norm greater than $2$.
For $n \ge 3$,
these are the cusps $C_1, \ldots C_k$ and $C$.
From $M_b$, we will construct a labelled {\em cusp graph}  $\mathcal{C}_b$ 
with vertices $\{v_1, v_2, ..., v_{k+1}\}$ as follows.
We label $v_j$ by $m_{i_j}$ for $1\leq j\leq k$ and $v_{k+1}$ by $n\sqrt{-1}$, 
the modulus of $C$. 
Then join $v_i$ and $v_{i'}$ by an edge if in $\mathcal{H}_b$, horoball neighbourhoods of the cusps corresponding to $v_i$ and $v_{i'}$ are tangent to each other.
In $\mathcal{H}_b$, the horoball neighbourhood of each cusp $C_j$ bumps only into $C_{j-1}$, $C_{j+1}$ (with index modulo $k$), and $C$.
Therefore, the graph $\mathcal{C}_b$  
consists of a cycle of length $k$ whose vertices are labelled cyclically by the $m_{i_j}$'s and one vertex labelled by $n$ joined to every vertex of the cycle. 

The special case when $b=111 \cdots 1$ and $k=0$ is slightly different. Then $M_b$ has one cusp $C$ of modulus
$n \sqrt{-1}$ as above, two knot components giving cusps $C_1,C_2$ of modulus $2n \sqrt{-1}$, and $2n$
cusps of modulus $\sqrt{-1}$. In this case, the cusp graph $\mathcal{C}_b$ is a 3-cycle.

\smallskip
Conversely given such a labelled cusp graph, we can reconstruct a binary cyclic word.
Therefore given two binary cyclic words $b_1$ and $b_2$ of length $n$,  the graph $\mathcal{C}_{b_1}$ is isomorphic to 
$\mathcal{C}_{b_2}$ as a graph with labelled vertices if and only if $b_1$ is related to $b_2$ by the natural action of $D_{n}$.
Moreover, by Remark \ref{horo_packing_invariant}, any isometry between manifolds $M_{b_1}$ and $M_{b_2}$ takes
the horoball packing $\mathcal{H}_{b_1}$ to $\mathcal{H}_{b_2}$.
Hence, the manifold $M_{b_1}$ is isometric to $M_{b_2}$ if and only if the graph $\mathcal{C}_{b_1}$ is isomorphic to $\mathcal{C}_{b_2}$ as a  graph with labelled vertices.
This completes the proof.
\end{proof}

\begin{thm}\label{thm.expo}
For each $n \ge 3$, there are at least $2^n/(2n)$ different hyperbolic 
link complements of volume $4nV_8$. 
\end{thm}
\begin{proof}
By theorem \ref{change_mfd}, there are at least as many hyperbolic link complements of volume $4nV_8$ as cyclic binary words of length $n$, up to the action of $D_n$. But there are $2^n$ based cyclic binary words of length $n$,
and each orbit under the $D_n$ action has at most $2n$ elements. So the result follows.
\end{proof}

\begin{rmk}
{\em For the sequence $v_n = 4nV_8$, this gives 
a logarithmic growth rate
$$\limsup_{n\to\infty} {\log  N_L(v_n) \over v_n} \ge
{\log 2 \over 4 V_8} \approx 0.0472962 .$$
Chesebro-BeBlois \cite{CD} construct examples of $2^n/2$
hyperbolic link complements of volume $w_n \approx 24.092184n + 2V_8$,
giving 
a logarithmic growth rate
$$\limsup_{n\to\infty} {\log  N_L(w_n) \over w_n} 
 \ge 0.0287706 .$$
}
\end{rmk}

\section{Some Open Questions} 

\begin{enumerate}

\item Find additional exact values of $N(v)$ or upper bounds on $N(v)$.

\ms
\item (Gromov \cite{Gr}, 1979) Is $N(v)$ locally bounded?

\ss
This reduces to the question of whether the number of manifolds with a given volume $v$
is uniformly bounded amongst all hyperbolic Dehn fillings on any given cusped hyperbolic 3-manifold. 

\ms
\item What is the largest volume $< v_\omega= 2.029883\ldots$
of a closed hyperbolic 3-manifold 
that does not arise from Dehn filling of $m004$ or $m003$?
(This would allow us to make the result of Theorem \ref{closed_unique} explicit.)

\ss
Experimental evidence suggests that the largest such volume is
$2.028853\ldots$
for the closed manifold $m006(-5,2)$.
(This is the largest such volume arising in the Hodgson-Weeks census
of low volume closed hyperbolic 3-manifolds with 
shortest closed geodesic of length $>0.3$,
and also for Dehn fillings on the ``magic manifold'' s776.)

\ms
\item (i) Are all hyperbolic Dehn fillings on $m004$ determined by their volumes,
amongst Dehn fillings on $m004$?\\
(ii) Are all hyperbolic Dehn fillings on $m003$ determined by their volumes, 
amongst Dehn fillings on $m003$?\\
(iii) There are some equalities  between volumes of Dehn fillings on $m003$ and $m004$,
observed in \cite{HMW}, and shown in equation (\ref{eq26}) above.
The Meyerhoff manifold of volume $0.981368\ldots$ also arises as Dehn filling on both:
$m004(5,1)=m003(-2,3) = m003(-1, 3)$.
Are these the only equalities between volumes of Dehn fillings on $m003$ and $m004$?

\ss
(These statements can be checked experimentally using SnapPy and Snap:
for instance they are true for all hyperbolic  Dehn fillings  
on $m003$ and $m004$
with $\text{volume} < 2.0289$.)

\ms
\item For which cusped hyperbolic 3-manifolds $N$ are there infinitely many hyperbolic Dehn fillings on $N$ that are uniquely determined by their volumes amongst Dehn fillings on $N$?

\ss
Our methods give some further results on this question when the cusp shape $\tau$  
lies in an imaginary quadratic number field.
But for $\tau$ algebraic of higher degree, our approach would require new results on   
the distribution of gaps in the values of the quadratic form 
$Q(a,b) =|a+ b \tau|^2$ for integers $a,b$. This question also arises
in the subject of ``Quantum Chaos'', and it is conjectured that the successive gaps 
should be randomly distributed according to a Poisson distribution
(see \cite{sarnak},  \cite{rudnick}, \cite{EMM}).
This would imply that arbitrarily large 2-sided gaps always exist.

\ms
\item  Are there sequences $x_i \to \infty$ with $N(x_i)=1$ for all $i$?

\ms
\item Are there sequences $x_i \to \infty$ such that the growth rate of the number $N_L(x_i)$ of link complements
with volume $x_i$ is faster than exponential, i.e with $\displaystyle\limsup_{i\to \infty} {\log N_L(x_i) \over \log x_i} = +\infty$?
\end{enumerate}


\begin{thebibliography}{99}

\bibitem{AD}
J. W. Aaber and N. M. Dunfield, Closed surface bundles of least volume, Algebr.  Geom. Topol. 10 (2010), 2315--2342. 

\bibitem{Ada} 	C. Adams, Thrice-punctured spheres in hyperbolic 3-manifolds, Trans. Amer. Math. Soc. 287 (1985), no. 2, 645-656. 

\bibitem{agol_dehn}     I. Agol, 
Bounds on exceptional Dehn filling. Geom. Topol. 4 (2000), 431--449. 

\bibitem{agol_2cusp}
I.~Agol, {The minimal volume orientable hyperbolic 2-cusped 3-manifolds}.
Proc. Amer. Math. Soc. {138} (2010), 3723--3732.

\bibitem{BGLS} M. Belolipetsky, T. Gelander, A. Lubotzky and A. Shalev,
Counting arithmetic lattices and surfaces, Annals of Math. (2) 172 (2010), no. 3,  2197--2221.

\bibitem{BGLM}  M. Burger, T. Gelander, A. Lubotzky and S. Mozes,
Counting hyperbolic manifolds, Geom. Funct. Anal. 12 (2002), 1161--1173.

\bibitem{buell} D. A. Buell, Binary quadratic forms:
Classical theory and modern computations, Springer-Verlag, New York, 1989.

\bibitem{CaHiWe} P. Callahan, M. Hildebrand and  J. Weeks, 
A census of cusped hyperbolic 3-manifolds, Math. Comp. 68 (1999), no. 225, 321--332. 

\bibitem{CM} C. Cao and G. R. Meyerhoff, 
The orientable cusped hyperbolic 3-manifolds of minimum volume,
Invent. Math. 146 (2001), no. 3, 451--478. 

\bibitem{carlip} S. Carlip,    
Dominant topologies in Euclidean quantum gravity,
Topology of the Universe Conference (Cleveland, OH, 1997),
Classical Quantum Gravity 15 (1998), no. 9, 2629--2638.

\bibitem{CD} E. Chesebro and J. DeBlois,
Algebraic invariants, mutation, and commensurability of link complements,
arXiv:1202.0765.

\bibitem{CFJR}
T. Chinburg, E. Friedman, K. Jones and A. Reid,
The arithmetic hyperbolic 3-manifold of smallest volume, 
Ann. Scuola Norm. Sup. Pisa Cl. Sci. (4) 30 (2001), no. 1, 1--40. 

\bibitem{cox} D. A. Cox, 
Primes of the form $x^2+ny^2$:
Fermat, class field theory and complex multiplication, Wiley, New York, 1989.

\bibitem{CDW} M. Culler, N. M. Dunfield, and J. R. Weeks, SnapPy, a computer program for studying 
the geometry and topology of 3-manifolds. Available at http://snappy.computop.org

\bibitem{DGST} N. M. Dunfield , S. Garoufalidis, A. Shumakovitch and M. Thistlethwaite, 
Behavior of knot invariants under genus 2 mutation, New York J. Math. 16 (2010), 99-123. 

\bibitem{EMM}
A. Eskin, G. Margulis and  S. Mozes, 
Quadratic forms of signature (2, 2) and eigenvalue spacings on rectangular 2-tori, 
Ann. of Math. (2) 161 (2005), no. 2, 679--725.

\bibitem{FMP1} R. Frigerio, B. Martelli and C. Petronio, Complexity and Heegaard genus of an infinite class of compact 3-manifolds, Pacific J. Math. 210 (2003), 283--297.

\bibitem{FMP2} R. Frigerio, B. Martelli and C. Petronio, Dehn filling of cusped hyperbolic 3-manifolds with geodesic boundary, J. Differ. Geom. 64 (2003), 425-455.

\bibitem{Fu}
M. Fujii,  Hyperbolic 3-manifolds with totally geodesic boundary which are decomposed into hyperbolic truncated tetrahedra, Tokyo J. Math. 13 (1990), no. 2, 353--373.

\bibitem{FP} D. Futer and J. S. Purcell, Links with no exceptional surgeries, Comment. Math. Helv. 82 (2007), no. 3, 629-664.

\bibitem{GMM1} D. Gabai, G. R.Meyerhoff,  P. Milley, 
{Minimum volume cusped hyperbolic three-manifolds},
{J. Amer. Math. Soc.} 22 (2009), {1157--1215}.

\bibitem{GMM2} D. Gabai, G. R.Meyerhoff,  P. Milley, 
{Mom technology and volumes of hyperbolic 3-manifolds},
{Comment. Math. Helv.} {86} (2011),  {145--188}.

\bibitem{G} O. Goodman, Snap, a computer program for studying arithmetic invariants of hyperbolic 3-manifolds. Available from http://www.ms.unimelb.edu.au/$\sim$snap

\bibitem{Gr} M. Gromov, Hyperbolic manifolds (according to Thurston and J{\o}rgensen), 
 Bourbaki Seminar, Vol. 1979/80, pp. 40--53, Lecture Notes in Math., 842, Springer, 
 Berlin-New York, 1981. 

\bibitem{HaWr} G. H. Hardy and E. M. Wright, An introduction to the theory of numbers,
Oxford University Press, 5th edition, 1979.

\bibitem{HHMP} D. Heard, C. Hodgson, B. Martelli, C. Petronio, Hyperbolic graphs of small complexity, 
Experiment. Math. 19 (2010), 211-236.

\bibitem{HMW} C. Hodgson, R. Meyerhoff and J. Weeks,
Surgeries on the Whitehead link yield geometrically similar manifolds, Topology '90 (Columbus, OH, 1990), 195--206, 
de Gruyter, Berlin, 1992. 

\bibitem{HK} C. Hodgson and S. Kerckhoff, The shape of hyperbolic Dehn surgery space,
Geom. Topol. {12} (2008), no. 2, 1033--1090. 

\bibitem{HW} C. Hodgson and J. Weeks, {Symmetries, isometries and length spectra of closed hyperbolic  three-manifolds},  Experiment. Math. 3 (1994),  261--274.

\bibitem{KojMiy}
S.~Kojima and Y.~Miyamoto, 
{The smallest hyperbolic 3-manifolds with totally geodesic boundary},
 J. Differential Geom. {34} (1991), no. 1, 175--192.

\bibitem{lack} M. Lackenby, Word hyperbolic Dehn surgery, Invent. Math. 140 (2000), no. 2, 243--282. 

\bibitem{Lac} M. Lackenby, The volume of hyperbolic alternating link complements, 
Proc. London Math. Soc. (3) 88 (2004), no. 1, 204-224. 
With an appendix by I. Agol and D. Thurston.

\bibitem{landau}
E. Landau, Elementary number theory.
Chelsea Publishing Co., New York, 1958.

\bibitem{Mas} H. Masai, On volume preserving moves on graphs and their applications, in preparation.

\bibitem{M} P. Milley, 
{Minimum volume hyperbolic 3-manifolds},
{J. Topol.} 2 (2009), {181--192}.

\bibitem{miyamoto}
Y.~Miyamoto, 
{Volumes of hyperbolic manifolds with geodesic boundary},
Topology {33} (1994), 613--629. 

\bibitem{NR}
W.~D.~Neumann and A.~W.~Reid,
 {Arithmetic of Hyperbolic Manifolds},
in {Topology '90, Proceedings of the Research Semester in Low
  Dimensional Topology at Ohio State},  de Gruyter,
  Berlin-New York, 1992, 273--310.

\bibitem{NZ} W. Neumann and D. Zagier, Volumes of hyperbolic three-manifolds,
Topology 24 (1985), no. 3, 307--332. 

\bibitem{Pur} J. Purcell, An introduction to fully augmented links, 
 Interactions between hyperbolic geometry, quantum topology and number theory, 205--220,
Contemp. Math., 541, Amer. Math. Soc., 2011. 
  
\bibitem{Rub}
D. Ruberman, Mutation and volumes of knots in $S^3$,
Invent. Math. 90 (1987), no. 1, 189--215.

\bibitem{rudnick}
Z. Rudnick,  What is... quantum chaos?,
Notices Amer. Math. Soc. 55 (2008), no. 1, 32--34. 

\bibitem{sarnak}
P. Sarnak,  Values at integers of binary quadratic forms, in 
Harmonic analysis and number theory (Montreal, PQ, 1996), 181--203, CMS Conf. Proc., 21, Amer. Math. Soc., 1997.

\bibitem{Th} W. Thurston, Geometry and topology of 3-manifolds, Lecture notes, Princeton University, 1978.

\bibitem{weeks_phd} J. R. Weeks, Hyperbolic structures on 3-manifolds, Ph.D. thesis, Princeton University, 1985.

\bibitem{We} J. R. Weeks, SnapPea: A computer program for creating and studying hyperbolic 3- manifolds. Available at http://www.geometrygames.org/SnapPea/

\bibitem{Wie} N. Wielenberg, Hyperbolic 3-manifolds which share a fundamental polyhedron, Riemann Surfaces and Related Topics: Proceedings of the 1978 Stony Brook Conference, Princeton Univ. Press, Princeton, N. J., 1981, pp. 505-513.

\bibitem{Zim} B. Zimmermann, A note on hyperbolic 3-manifolds of the same volume, Monaths. Math., 110 (1990), pp. 321-327.

\end{thebibliography}
\end{document}